\documentclass[reqno]{amsart}
\usepackage{amssymb,latexsym}
\usepackage{amsmath,color}
\usepackage{amsthm}
\usepackage{epsfig}
\usepackage{graphicx}
\usepackage{hyperref}
\usepackage{titletoc}
\usepackage{palatino,mathpazo}
\numberwithin{equation}{section}
\newtheorem{theorem}{Theorem}[section]

\newtheorem{lemma}[theorem]{Lemma}

\newtheorem{Theorem}{Theorem}[section]
\newtheorem{Thm}{Theorem}[section]

\theoremstyle{definition}

\def\XXint#1#2#3{{\setbox0=\hbox{$#1{#2#3}{\int}$}
     \vcenter{\hbox{$#2#3$}}\kern-.5\wd0}}

\def\B{\mathbb{R}}
\def\l{\lambda}

\def\v{\varepsilon}
\def\om{\overline{\Omega}}
\def\m{\mathrm{C}_{\mathrm{loc}}}

\begin{document}
\title[Gel'fand problem]{Local uniqueness of $m$-bubbling sequences for the Gel'fand equation.}

\author{Daniele Bartolucci$^{(\dag)}$}
 \address{Daniele Bartolucci, Department of Mathematics, University of Rome {\it "Tor Vergata"},  Via della ricerca scientifica n.1, 00133 Roma,
Italy.}
\email{bartoluc@mat.uniroma2.it}

\author{Aleks Jevnikar}
 \address{Aleks Jevnikar, Department of Mathematics, University of Rome {\it "Tor Vergata"},  Via della ricerca scientifica n.1, 00133 Roma,
Italy.}
\email{jevnikar@mat.uniroma2.it}

\author{Youngae Lee}
\address{Youngae ~Lee,~National Institute for Mathematical Sciences, 70 Yuseong-daero 1689 beon-gil, Yuseong-gu, Daejeon, 34047, Republic of Korea}
\email{youngaelee0531@gmail.com}

\author{ Wen Yang$^{(\ddag)}$}
\address{ Wen ~Yang,~Wuhan Institute of Physics and Mathematics, Chinese Academy of Sciences, P.O. Box 71010, Wuhan 430071, P. R. China}
\email{math.yangwen@gmail.com}


%\thanks{$^{(1)}$Daniele Bartolucci, Department of Mathematics, University
%of Rome {\it "Tor Vergata"}, \\  Via della ricerca scientifica n.1, 00133 Roma,
%Italy. e-mail:bartoluc@mat.uniroma2.it}

\thanks{$^{(\dag)}$Research partially supported by FIRB project "{\em
Analysis and Beyond}",  by PRIN project 2012, ERC PE1\_11,
"{\em Variational and perturbative aspects in nonlinear differential problems}", and by the Consolidate the Foundations
project 2015 (sponsored by Univ. of Rome "Tor Vergata"),  ERC PE1\_11,
"{\em Nonlinear Differential Problems and their Applications}".}
\thanks{$^{(\ddag)}$ Research partially supported by CAS Pioneer Hundred Talents Program (Y8S3011001).}

\begin{abstract}
We consider the  Gel'fand problem,
\begin{equation*}
\begin{cases}
\Delta w_{\v}+\v^2he^{w_{\v}}=0\quad&\mathrm{in}\quad\Omega,\\
\\
w_{\v}=0\quad&\mathrm{on}\quad\partial\Omega,
\end{cases}
\end{equation*}
where $h$ is a nonnegative function in ${\Omega\subset\mathbb{R}^2}$. {Under suitable assumptions on $h$ and $\Omega$,
we prove the local uniqueness of $m-$bubbling solutions for any $\varepsilon>0$ small enough.}

\end{abstract}
\maketitle

\section{\bf Introduction}
We are concerned with the Gel'fand problem,
\begin{equation}
\label{1.1e}
\begin{cases}
-\Delta w_{n}=\v^2_nhe^{w_n}\quad &\mathrm{in}\quad\Omega,\\
\\
w_{n}=0~&\mathrm{on}\quad\partial\Omega,
\end{cases}
\end{equation}
where ${\Omega\subset\mathbb{R}^2}$ is a {smooth  and bounded domain,  $\lim\limits_{n\to+\infty}\v_n\to0$, and}
\begin{equation}
\label{1.h}
h(x)=\hat h(x)\exp(-4\pi\sum_{i=1}^\ell\alpha_iG(x,p_i))\geq 0,\ \ {\hat h>0, \ \hat{h}\in C^{\infty}(\overline\Omega).}
\end{equation}
{Here} $p_i$'s are distinct points in $\Omega,$ $\alpha_i>-1$ for any $i=1,\cdots,\ell$   and $G(x,p)$ is the Green function satisfying,
\begin{equation*}
-\Delta G(x,p){=}\delta_p\quad \mathrm{in}\quad \Omega,\quad G(x,p)=0\quad \mathrm{on}\quad \partial\Omega.
\end{equation*}
The equation in \eqref{1.1e} and its mean field type analogue \eqref{1m} below, have a long history in pure and applied mathematics, ranging from conformal geometry, thermal ignition models, kinetic and mean field models
in statistical mechanics, chemiotaxis dynamics and gauge field theories, see for example \cite{beb, clmp1, KW, suzC, tar, w} and references therein.
Therefore a huge work has been done to understand these equations. The literature is so large that it is impossible to provide a complete list and we just refer to
to \cite{BDeM, BdM2, BLin3, bm, cli1, CLin1, CLin2, CLin5, EGP, KMdP, yy, ls, Mal2, NS90, OSS, suz} and the references quoted therein.
Among many other things, an interesting property of these problems is the lack of compactness \cite{bm} which in turn causes non uniqueness of
solutions for $\v>0$. As a consequence the global bifurcation diagram is in general very rich and complicated \cite{CLin2}, \cite{NS90} and solutions may be degenerate.
Here we say that a solution of \eqref{1.1e} is degenerate if the corresponding linearized problem {(see \eqref{1.8e} below)} admits a non trivial solution.
This is why it is important to understand where we can recover uniqueness and non degeneracy in the bifurcation diagram.

{In this paper, we consider a sequence of bubbling solutions  ${ w_{n}}$  of \eqref{1.1e} with $n\to+\infty$.}

\medskip
\noindent {\bf Definition 1.1.} \emph{A sequence of solutions   {${ w_{n}}$} of \eqref{1.1e}, is said to be an $m-$bubbling {(or blow up)} sequence {if}
\begin{equation*}
\v_n^2he^{{ w_{n}}} \rightharpoonup 8\pi\sum_{j=1}^m\delta_{q_j}\quad\mathrm{as}\quad n\to+\infty,
\end{equation*}
weakly in the sense of measures in $\Omega,$ where {$\{q_1,\cdots,q_m\}\subset\Omega$} are $m$ distinct points satisfying $\{q_1,\cdots,q_m\}\cap\{p_1,\cdots,p_\ell\}=\emptyset$. The points $q_j$ are said to be the blow up points and $\{q_1,\cdots,q_m\}$ the blow up set.}
\medskip

\medskip
\noindent {\bf Remark 1.1.} \emph{It is well known  \cite{yy, ls, NS90}, that if $w_{n}$ of \eqref{1.1e} is an $m-$bubbling sequence, then $\v_n^2\int_{\Omega}he^{w_n}\to 8m\pi$ as $n\to+\infty.$}
\medskip

By letting $\l_n:=\v_n^2\int_{\Omega}he^{{ w_{n}}}$, we see that \eqref{1.1e} takes the form of the well known mean field equation \cite{clmp1},
\begin{equation}
\label{1m}
\begin{cases}
-\Delta {u_{n}=\l_n\dfrac{he^{u_{n}}}{\int_{\Omega}he^{u_{n}}}}\quad &\mathrm{in}\quad\Omega,\\
\\
{u_{n}}=0~&\mathrm{on}\quad\partial\Omega.
\end{cases}
\end{equation}
%\noindent Non-degeneracy and uniqueness are rather delicate problems which have been discussed in \cite{suz} for \eqref{1.1e}. However, we do not know of any general non-degeneracy result for bubbling solutions of \eqref{1.1e}, with the exception of the results in \cite{GlGr} where the authors are concerned with the case $m=1$.
%In this paper, we consider a generic $m-$bubbling sequence, which therefore satisfy
%\begin{equation}
%\label{1.1e}
%\begin{cases}
%\Delta {w_{n} }+\v_n^2he^{{w_{n} }}=0\quad &\mathrm{in}\quad\Omega,\\
%\\
%{w_{n} }=0 &\mathrm{on}\quad\partial\Omega,
%\end{cases}
%\end{equation}
%for any $n\in\mathbb{N}$, and prove that ${w_{n} }$ is non-degenerate, provided that $n$ is sufficiently large and $h$ and $\Omega$ satisfy some other non-degeneracy assumption.
%\medskip
%Equation \eqref{1.1e} has a close relation to the following Mean Field equation
%\begin{equation}
%\label{1m}
%\begin{cases}
%\Delta u+\l\frac{he^{u}}{\int_{\Omega}he^{u}}=0\quad &\mathrm{in}\quad\Omega,\\
%\\
%u=0 &\mathrm{on}\quad\partial\Omega.
%\end{cases}
%\end{equation}
%\medskip
%In the recent works \cite{BJLY,BJLY2}, {when $\l_n$ is fixed,} the authors have studied the local uniqueness and non-degeneracy for the bubbling solutions of \eqref{1m}.
Let us shortly discuss some results about \eqref{1m}. Let
$$R(x,y)=G(x,y)+\frac{1}{2\pi}\log|x-y|,$$
be the regular part of the Green function $G(x,y)$. For fixed
\begin{equation*}
{\bf q}=(q_1,\cdots,q_m)\in\Omega\times\cdots\times\Omega,
\end{equation*}
we set,
\begin{equation}
\label{1.2}
G_j^*(x)=8\pi R(x,q_j)+8\pi\sum_{l\neq j} G(x,q_l),
\end{equation}
\begin{equation}
\label{1.3}
l({\bf q})=\sum_{j=1}^m\left(\Delta\log h(q_j)\right)h(q_j)e^{G_j^*(q_j)},
\end{equation}
and
\begin{equation}
\label{1.4}
f_{\mathbf{q},j}(x)=8\pi\left[R(x,q_j)- R(q_j,q_j)+\sum_{l\neq j}(G(x,q_l)-G(q_j,q_l))\right]+\log \frac{h(x)}{ h(q_j)}.
\end{equation}
We will denote by $B_r(q)$ the ball of radius $r$ centred at $q\in\Omega$. For the case $m\geq2$ we fix a constant $r_0\in(0,\frac12)$ and a family of open sets $\Omega_j$ satisfying $\Omega_l\cap\Omega_j=\emptyset$ if $l\neq j$, $\bigcup_{j=1}^m\om_j=\om$, $B_{2r_0}(q_j)\subset\Omega_j,~j=1,\cdots,m$. Then, let us define,
\begin{equation}
\label{1.5}
D(\mathbf{q})=\lim_{r\to0}\sum_{j=1}^m h(q_j)e^{G_j^*(q_j)}\left(\int_{\Omega_j\setminus B_{r_j}(q_j)}
e^{\Phi_j(x,\mathbf{q})} \mathrm{d}x-\frac{\pi}{r_j^2}\right),
\end{equation}
where $\Omega_1=\Omega$ if $m=1$, $r_j=r\sqrt{8h(q_j)e^{G_j^*}(q_j)}$ and
\begin{equation}
\label{1.6}
\Phi_j(x,\mathbf{q})=8\pi\sum_{l=1}^mG(x,q_l)-G_j^*(q_j)+\log h(x)-\log h(q_j).
\end{equation}

{ To describe} the location of the blow up points, we introduce the following function.  For $(x_1,\cdots,x_m)\in\Omega\times\cdots\times\Omega$ we define,
\begin{equation}
\label{1.7}
f_m(x_1,x_2,\cdots,x_m)=\sum_{j=1}^{m}\Big[\log(h(x_j))+4\pi R(x_j,x_j)\Big]+4\pi\sum_{l\neq j}G(x_l,x_j),
\end{equation}
and let $D_{\Omega}^2f_m$ be its Hessian matrix on $\Omega$. The function $f_m$ is also known in literature as the $m-$vortex Hamiltonian \cite{New}.
It is also well known \cite{CLin1, NS90} that the blow up points vector ${\bf q}=(q_1,\cdots,q_m)$ is a critical point of $f_m$.
{In this paper, we deal with critical points of the function $f_m$ whose Hessian matrix $D_{\Omega}^2f_m$ is non-degenerate.}
\medskip

{In view of Remark 1.1, we see that, if $u_n$ is an $m$-bubbling sequence of \eqref{1m}, then $\l_n\to8m\pi$.
We say that local uniqueness of $m$-bubbling sequences of \eqref{1m} holds if the following is true:\\
for a fixed critical point ${\bf q}=(q_1,\cdots,q_m)$ of $f_m$, if there exists two $m$-bubbling sequences $u_n^{(1)},u_n^{(2)}$ of \eqref{1m}
with the same $\l_n$ and  whose blow up set is $\{q_1,\cdots,q_m\}$,
then, if $\l_n$ is close enough to $8\pi m$, it holds $u_n^{(1)}\equiv u_n^{(2)}$.\\
Actually, Lin and Yan in \cite{ly2} have initiated the study of the local uniqueness of $m$-bubbling sequences for the
Chern-Simons-Higgs equation. Inspired by that approach, in \cite{BJLY,BJLY2} the authors of this work proved
local uniqueness and non-degeneracy of $m$-bubbling sequences  of \eqref{1m}. More precisely, we have the following:}

\begin{Thm}[\cite{BJLY,BJLY2}]
\label{th1a.1}
Let $\mathbf{q}=(q_1,\cdots,q_m)$ be a critical point of $f_m$ such that $q_j\in\Omega\setminus\{p_1,\cdots,p_\ell\}$, $j=1,\cdots,m$
and $det(D_{\Omega}^2f_m(\mathbf{q}))\neq0$. Assume that either $\ell(\mathbf{q})\neq0$ or $D(\mathbf{q})\neq0$. Then the following facts hold true:

\noindent (i) Local uniqueness with respect to $\l_n$: Let $u_{n}^{(1)}$ and $u_{n}^{(2)}$ be  two sequence of solutions of \eqref{1m}
with  fixed $\l_n$. If $\l_n$ is sufficiently close to $8m\pi$, then $u_{n}^{(1)}\equiv u_{n}^{(2)}$.

\noindent (ii) Non-degeneracy with respect to $\l_n$: If $\l_n$ is  sufficiently close to $8m\pi$, then the linearized problem for \eqref{1m},
\begin{equation}
\label{1.8l}
\begin{cases}
\Delta\phi_n+\lambda_n\frac{he^{u_n}}{\int_{\Omega}he^{u_n}d y}\left(\phi_n-\frac{\int_{\Omega}he^{u_n}\phi_n d y}{\int_{\Omega}he^{u_n}d y}\right)=0 \quad &\mathrm{in}\quad \Omega, \\
\\
\phi_n=0\ &\mathrm{on}\quad \partial\Omega,
\end{cases}
\end{equation}
admits only the trivial solution $\phi_n\equiv 0$.
\end{Thm}

We note that for a fixed constant $\l_n$ in Theorem 1A,  the corresponding
$$(\v_n^{(i)})^2=\frac{\l_n}{\int_{\Omega}he^{u_n^{(i)}}}$$
might be different for $i=1,2.$ Interestingly enough, it turns out that if $\v_n^2=\frac{\l_n}{\int_{\Omega}he^{u_n}}$ is regarded as the main fixed parameter
instead of $\l_n$, then the quantities $D({\bf q})$ and $l({\bf q})$ are no longer important as they happen to be if $\l_n$ is held fixed.
This difference is caused by the difference in the linearized operators. Indeed, for a fixed $\l_n$, the corresponding linearized problem is \eqref{1.8l},
but for a fixed $\v_n$, we have the following linearized problem:
\begin{equation}
\label{1.8e}
\begin{cases}
\Delta{\phi_n}+\v_n^2he^{{w_{n} }}{\phi_n}=0~\quad &\mathrm{in}\quad\Omega,\\
\\
{\phi_n}=0~&\mathrm{on}\quad\partial\Omega.
\end{cases}
\end{equation}
The non-degeneracy of the latter problem with respect to $\v_n>0$ is already well known \cite{GG,GOS,OSS}:

\begin{Thm}
[\cite{GG,GOS,OSS}]
\label{th1a.2}
Let ${w_{n}}$ be an $m-$bubbling sequence of \eqref{1.1e} whose blow up set is $\{q_1,\cdots,q_m\}\subset \Omega\setminus\{p_1,\cdots,p_\ell\}$,
where ${\bf q}=(q_1,\cdots,q_m)$ is a critical point of $f_m$ such that $det(D_{\Omega}^2f_m({\bf q}))\neq0$.
If $\v_n>0$ is  sufficiently small, then the linearized problem \eqref{1.8e} admits only the trivial solution ${\phi_n}\equiv0.$
\end{Thm}

%\noindent {{\bf Remark 1.2.} Domains non smooth as in \cite{CCL}, $h_>h_n.$}
%\medskip

On the other hand, the local uniqueness of $m$-bubbling sequences for \eqref{1.1e} has remained a long standing open problem. Indeed, compared to the non-degeneracy argument, one has to face a truly new difficulty when comparing different bubbling sequences for \eqref{1.1e}, see the discussion later on. The aim of the present paper is to overcome this difficulty and solve the local uniqueness problem. More precisely, motivated by Theorem 1B and \cite{ly2,BJLY}, it is natural to ask whether or not the local uniqueness of $m$-bubbling sequences for \eqref{1.1e} with respect to
$\v_n>0$ holds as well even if we do not assume that either $\ell(\mathbf{q})\neq0$ or  $D(\mathbf{q})\neq0$.
Indeed, it turns out that, whenever there exist two solutions $w_{\v_n}^{(1)}$ and $w_{\v_n}^{(2)}$ of \eqref{1.1e} sharing the same value of $\v_n>0$,
and the same blow up set $\{q_1,\cdots,q_{m}\}\subset\Omega\setminus\{p_1,\cdots,p_{\ell}\}$, then we
only need the non-degenerate assumption of the Hessian matrix $D_{\Omega}^2f_m$ at $\{q_1,\cdots,q_{m}\}$ to prove that $w_{\v_n}^{(1)}=w_{\v_n}^{(2)}$ for any
$\v_n$ small enough.

\begin{theorem}
\label{th1.2}
Let $w_{\v_n}^{(i)},~i=1,2$ be two $m-$bubbling sequences of \eqref{1.1e}, {with the same $\v_n>0$ and the same}
blow up set $\{q_1,\cdots,q_{m}\}\cap\{p_1,\cdots,p_{\ell}\}=\emptyset$,  where ${\bf q}$ is a critical point of $f_m$ and assume that
$D_{\Omega}^2f_m$ is non-degenerate.  {If $\v_n>0$ is sufficiently small, then $w_{\v_n}^{(1)}=w_{\v_n}^{(2)}.$}
\end{theorem}

{Compared to Theorem \ref{th1a.1}, we note that in Theorem \ref{th1.2}, for $\v_n>0$ fixed, then the corresponding
values of $\lambda_n^{(i)}=\varepsilon_n^2\int_\Omega h e^{w_n^{(i)}}dx$ may be different for $i=1,2$.}

To prove Theorem \ref{th1.2} we will analyse the asymptotic behavior of the normalized difference,
\begin{equation*}
{\xi_n=\frac{w_{\v_n}^{(1)}-w_{\v_n}^{(2)}}{\|w_{\v_n}^{(1)}-w_{\v_n}^{(2)}\|_{L^{\infty(\Omega)}}}.}
\end{equation*}
By using the refined estimates in \cite{CLin1}, we will see that near each blow up point $q_j$, and after a suitable scaling, then
$\xi_n$ converges to an entire solution of the linearized problem associated to the Liouville equation:
\begin{equation}
\label{1.9}
\Delta v+e^v=0\quad\mathrm{in}~\B^2,\quad \int_{\B^2}e^v<+\infty.
\end{equation}
The solutions of \eqref{1.9} are classified and take the following form, see \cite{cli1},
\begin{equation}
\label{1.10}
v(z)=v_{\mu,a}(z)=\log\frac{8e^{\mu}}{(1+e^{\mu}|z+a|^2)^2},\quad \mu\in\mathbb{R},\quad a=(a_1,a_2)\in\mathbb{R}^2.
\end{equation}
The linearized operator $L$ relative to $v_{0,0}$ is determined by,
\begin{equation}
\label{1.11}
L\phi:=\Delta\phi+\frac{8}{(1+|z|^2)^2}\phi\quad\mathrm{in}\quad \B^2.
\end{equation}
It is well known that the kernel of $L$ has real dimension 3 with eigenfunctions $Y_0,Y_1,Y_2$, where,
\begin{equation*}
\begin{aligned}
Y_0(z) = \frac{1-|z|^2}{1+ |z|^2}=\frac{\partial v_{\mu,a}}{\partial \mu}\Big|_{(\mu,a)=(0,0)},\\
Y_1(z) = \frac{z_1}{1+ |z|^2}=-\frac{1}{4}\frac{\partial v_{\mu,a}}{\partial a_1}\Big|_{(\mu,a)=(0,0)},\\
Y_2(z) = \frac{z_2}{1+ |z|^2}=-\frac{1}{4}\frac{\partial v_{\mu,a}}{\partial a_2}\Big|_{(\mu,a)=(0,0)}.
\end{aligned}
\end{equation*}
After a suitable scaling, we shall see that the projections of ${\xi_n}$ on $Y_0,Y_1,Y_2$
completely describe the behavior of ${\xi_n}$ near the blow up point.
In addition we can show that the projections of ${\xi_n}$ on $Y_0$ associated to each blow up point coincides
and that ${\xi_n}$ will converge to the corresponding (uniquely defined) Fourier coefficient far away from blow up points.\\
By using {a suitably defined Pohozaev identity} and in view of the non-degeneracy of $D_{\Omega}^2f_m$,
we will show that the projections on the translation kernels ($Y_1,Y_2$) associate to each blow up point is zero.
On the other side, for the projection on $Y_0$, we notice that after normalization and away from the blow up points,
the limit of the function ${\xi_n}$ tends to some harmonic function in $\Omega$ with Dirichlet boundary condition.
As a consequence, we will see that ${\xi_n}$ converges to $0$ outside the blow up area, which proves that the projection on $Y_0$ is $0$.
Therefore all those projections vanish which yield the desired result. We point out that this is different from the previous works
\cite{ly2} and \cite{BJLY,BJLY2}, where one is bound to use the assumptions about $l({\bf q})$ and $D({\bf q})$ to show that
the projection along $Y_0$ vanishes.\\

This paper is organized as follows. In section 2, we review some known sharp estimates for blow up solutions of \eqref{1m}.
In section 3 we prove Theorem \ref{th1.2} and leave some technical results in the Appendix.

%\bigskip
%
%\begin{center}
%{\bf Notation}
%\end{center}
%
%Throughout the paper we set the following notation and convention for convenience:
%\vspace{0.25cm}
%
%1. $r$ always denotes a small number, which is allowed to vary among different formulas and even within the same lines.
%\vspace{0.25cm}
%
%2. Denote by $B_r(q)$ the ball of radius $r$ centred at $q\in\Omega$. For $m\geq 2$ we fix a constant $r_0\in(0,\frac12)$ and a family of open sets %$\Omega_j$ satisfying $\Omega_l\cap\Omega_j=\emptyset$ if $l\neq j$, $\bigcup_{j=1}^m\om_j=\om$, $B_{2r_0}(q_j)\subset\Omega_j,~j=1,\cdots,m$.

\bigskip
\section{Preliminaries}
In this section we shall list some of the results which will be used below. By setting,
\begin{equation}
\label{2.1}
\l_n=\v_n^2\int_{\Omega}he^{{w_{n} }}.
\end{equation}
we see that $u_n{:=w_{n} }$ is an $m-$bubbling sequence of
\begin{equation}
\label{2.2}
\begin{cases}
\Delta u_n+\l_n\frac{he^{u_n}}{\int_{\Omega}he^{u_n}}=0\quad &\mathrm{in}\quad\Omega,\\
\\
u_{n}=0&\mathrm{on}\quad\partial\Omega,
\end{cases}
\end{equation}
which blows up at $q_j\notin\{p_1,\cdots,p_{\ell}\},~j=1,\cdots,m$ and in particular $\l_n\to8\pi m$, see Remark 1.1.
Next let us state various well known and sharp estimates for these blow up solutions of \eqref{2.2}. Let
\begin{equation*}
\tilde u_n=u_n-\log\left(\int_{\Omega}he^{u_n}\right),
\end{equation*}
then it is easy to see that,
\begin{equation}
\label{2.3}
\Delta\tilde u_n+\l_nh(x)e^{\tilde u_n(x)}=0\quad\mathrm{in}\quad\Omega,\quad\mathrm{and}\quad \int_{\Omega}he^{\tilde{u}_n}=1.
\end{equation}
We denote by,
\begin{equation}
\label{2.4}
\mu_n=\max_{\Omega}\tilde u_n, \quad \mu_{n,j}=\max_{B_{r_0}(q_j)}\tilde u_n=\tilde u_n(x_{n,j})\quad \mathrm{for}~j=1,\cdots,m.
\end{equation}

To give a good description of the bubbling solution around a blow up point, we define,
\begin{equation}
\label{2.5}
U_{n,j}(x)=\log\dfrac{e^{\mu_{n,j}}}{\left(1+\frac{\l_nh(x_{n,j})}{8}e^{\mu_{n,j}}|x-x_{n,j,*}|^2\right)^2},\quad x\in\B^2,
\end{equation}
where the point $x_{n,j,*}$ is chosen to satisfy,
\begin{equation*}
\nabla U_{n,j}(x_{n,j})=\nabla\log h(x_{n,j}),
\end{equation*}
and it is not difficult to check that
\begin{equation}
\label{2.6}
|x_{n,j}-x_{n,j,*}|=O(e^{-\mu_{n,j}}).
\end{equation}
In $B_{r_0}(x_{n,j})$, we use the following term $\eta_{n,j}$ to capture the difference between the genuine solution $\tilde u_n$ and the approximate bubble $U_{n,j}$:
\begin{equation}
\label{2.7}
\eta_{n,j}(x)=\tilde u_n-U_{n,j}(x)-\left(G_j^*(x)-G_j^*(x_{n,j})\right),\quad x\in B_{r_0}(x_{n,j}).
\end{equation}
Far away from the blow up points, i.e., $x\in\om\setminus\bigcup_{j=1}^mB_{r_0}(q_j)$, $\tilde u_n$ is well approximated by a sum of Green's function,
\begin{equation}
\label{2.8}
\phi_n(x)=\tilde u_n(x)-\sum_{j=1}^m\rho_{n,j}G(x,x_{n,j})-\tilde{u}_{n,0},
\end{equation}
where here and in the rest of this article, we set
\begin{equation}
\label{2.9}
\tilde u_{n,0}=\tilde u_n|_{\partial \Omega}\,,
\end{equation}
and the local mass associated to each blow up point $q_j,~1\leq j\leq m$ is defined by
\begin{equation}
\label{2.10}
\rho_{n,j}=\l_n\int_{B_{r_0}(q_j)}he^{\tilde u_n}\mathrm{d}y.
\end{equation}

The following refined estimates derived by Chen and Lin \cite{CLin1} concerning $\eta_{n,j},~1\leq j\leq n$ and $\phi_n$, play a crucial role in our argument.
\begin{Theorem}[\cite{CLin1}]
\label{th2.a}
Let $u_n$ be an $m$-bubbling sequence of  \eqref{2.2} which blows up at the points $q_j\notin\{p_1,\cdots,p_\ell\},~j=1,\cdots,m$.
Then we have:
\begin{itemize}
  \item [(a)] $\eta_{n,j}=O(\mu_{n,j}^2e^{-\mu_{n,j}})$ on $B_{r_0}(x_{n,j}),$
  \item [(b)] $\phi_n=o(e^{-\frac{\mu_{n,j}}{2}})$ in $C^1(\om\setminus\bigcup_{j=1}^mB_r(q_j)).$
\end{itemize}
\end{Theorem}

The interaction between bubbles relative to different blow up points, the difference between each local mass and $8\pi$ and the difference
between the parameter $\l_n$ and $8\pi m$ play an essential role in the understanding of the blow up behavior. We present these estimates as follows:
\begin{Theorem}
\label{th2.b}
Suppose that the assumptions in Theorem \ref{th2.a} hold, then for any $j=1,\cdots,m$ we have,
\begin{itemize}
  \item [$(i)$] $e^{\mu_{n,j}}h^2(x_{n,j})e^{G_j^*(x_{n,j})}
      =e^{\mu_{n,1}}h^2(x_{n,1})e^{G_1^*(x_{n,1})}(1+O(e^{-\frac{\mu_{n,1}}{2}})),$
  \item [$(ii)$] $\mu_{n,j}+\tilde{u}_{n,0}+2\log\left(\frac{\l_nh(x_{n,j})}{8}\right)+G_j^*(x_{n,j})  ={-\frac{2}{ \l_{n}h(x_{n,j)}}(\Delta \log h(x_{n,j} ))(\mu_{n,j} )^2e^{-\mu_{n,j} }}$

  \hfill${+O(\mu_{n,j}e^{-\mu_{n,j} }),}\quad\quad\quad\quad\quad\quad\quad\ \  $
  \item [$(iii)$] $\nabla\left(\log h(x)+G_j^*(x_{n,j})\right)\big|_{x=x_{n,j}}=O(\mu_{n,j}e^{-\mu_{n,j}})$, \ \
      $|x_{n,j}-q_j|=O(\mu_{n,j}e^{-\mu_{n,j}}),$
  \item [$(iv)$] $|\mu_n-\mu_{n,j}|\leq c$ for $j=1,\cdots,m$, $|\tilde u_n(x)+\mu_n|\leq c_{r_0}$ for $x\in\om\setminus\bigcup_{j=1}^mB_{r_0}(q_j),$
  \item [$(v)$] $\rho_{n,j}-8\pi=O(\mu_{n,j}e^{-\mu_{n,j}})$,
  \item [$(vi)$] for a fixed small constant $r>0$,
  \begin{equation}
  \label{2.11}
  \begin{aligned}
  \l_n-8\pi m=&\frac{2l({\bf q})e^{-\mu_{n,1}}}{mh^2(q_1)e^{G_j^*(q_1)}}\left(\mu_{n,1}+\log(\l_nh^2(q_1){e^{G_1^*(q_1)}r^2})-2\right)\\
  &+\frac{8e^{-\mu_{n,1}}}{\pi mh^2(q_1)e^{G_1^*(q_1)}}\left(D({\bf q})+O(r^{\sigma})\right)\\&+O(\mu_{n,1}^2e^{-\frac32\mu_{n,1}})
  +O(e^{-(1+\frac{\sigma}{2})\mu_{n,1}}),
  \end{aligned}
  \end{equation}
  where $\sigma>0$ is a positive number such that $\hat h\in C^{2,\sigma}(\om).$
\end{itemize}
\end{Theorem}

\noindent {\bf Remark 2.1.} \emph{The conclusions (i)-(v) and the estimate up to the order $\mu_ne^{-\mu_n}$ on $\lambda_n-8\pi m$
have been found in \cite{CLin1}. More recently \cite{BJLY} the authors of this paper derived the higher order estimate in (vi).
We remark that, although we will use (vi), we will not make use of the fact that either $l({\bf q})\neq 0$ or $D({\bf q})\neq 0$, see \eqref{a1.3} below.}

\bigskip
\section{Proof of Theorem \ref{th1.2}.}
To prove Theorem \ref{th1.2} we argue by contradiction and assume that \eqref{1.1e} has two different solutions $w_{n}^{(1)}$ and $w_{n}^{(2)}$ with the same $\v_n$, which blows up at $q_j\not\in\{p_1,\cdots,p_{\ell}\},~j=1,\cdots,m$. We set
\begin{equation}
\label{4.1}
{\l_{n}^{(i)}}={\v_n^2}\int_{\Omega}he^{w_n^{(i)}},\quad i=1,2,
\end{equation}
then $u_{n}^{(i)}:=w_{n}^{(i)},~i=1,2$ are the $m-$bubbling sequence of
\begin{equation}
\label{4.2}
\begin{cases}
\Delta u_{n}^{(i)}+{\l_{n}^{(i)}}\frac{he^{u_n^{(i)}}}{\int_{\Omega}he^{u_n^{(i)}}}=0\quad&\mathrm{in}\quad\Omega,\\
\\
u_{n}^{(i)}=0\quad&\mathrm{on}\quad \partial\Omega,
\end{cases}
\end{equation}
respectively, which blow up at $q_j\not\in\{p_1,\cdots,p_{\ell}\},~j=1,\cdots,m$. {By Remark 1.1,} we have ${\l_{n}^{(i)}} \to 8\pi m.$
We shall use $x_{n,j}^{(i)},\mu_n^{(i)},\mu_{n,j}^{(i)},\tilde u_n^{(i)},U_{n,j}^{(i)},\eta_{n,j}^{(i)}, {x_{n,j,*}^{(i)}},\phi_{n}^{(i)},\rho_{n,j}^{(i)}$
to denote $x_{n,j},\mu_n,\mu_{n,j},\tilde u_n,U_{n,j},\eta_{n,j},x_{n,j,*},\phi_{n},\rho_{n,j}$, as defined in section 2,
corresponding to $u_n^{(i)},~i=1,2$ respectively. In order to show that $u_{n}^{(1)}=u_n^{(2)}$, it is enough to show that,
\begin{equation}
\label{4.3}
\tilde u_{n}^{(1)}+{\log\l_{n}^{(1)}}=\tilde u_{n}^{(2)}+\log{\l_{n}^{(2)}}.
\end{equation}
Indeed, it is not difficult to see from \eqref{4.1} that,
\begin{equation}
\label{4.4}
\tilde u_{n}^{(i)}+\log{\l_{n}^{(i)}}=u_{n}^{(i)}+2\log \v,\quad i=1,2.
\end{equation}

Therefore, whenever we prove \eqref{4.3}, then $u_{n}^{(1)}=u_{n}^{(2)}$ immediately follows from  \eqref{4.4}. It is useful to set,
\begin{equation}
\label{4.5}
\hat u_n^{(i)}=\tilde u_{n}^{(i)}+\log{\l_{n}^{(i)}} ,\quad \hat\mu_{n,j}^{(i)}=\mu_{n,j}^{(i)}+\log{\l_{n}^{(i)}} ,~j=1,\cdots,m,
\end{equation}
so that in particular we find,
\begin{equation}
\label{4.6}
\hat u_n^{(i)}=u_n^{(i)}+2\log\v,\quad i=1,2.
\end{equation}
Next, let us study the term $\hat u_n^{(1)}-\hat u_n^{(2)}$. First of all we obtain an estimate about $\|\hat u_n^{(1)}-\hat u_n^{(2)}\|_{L^\infty(\Omega)}$.
\begin{lemma}
\label{le4.1}
$(i)$ $|\hat\mu_{n,j}^{(1)}-\hat\mu_{n,j}^{(2)}|=O\big(\sum_{i=1}^2\hat\mu_{n,j}^{(i)}e^{-\hat\mu_{n,j}^{(i)}}\big)$
for all $1\le j\le m$.

$(ii)$ $\|\hat{u}_n^{(1)}-\hat{u}_n^{(2)}\|_{L^\infty(\Omega)}=O\big(\sum_{i=1}^2\hat\mu_{n,1}^{(i)}e^{-\frac{\hat\mu_{n,1}^{(i)}}{2}}\big).$
\end{lemma}

\begin{proof}
$(i)$ In view of Theorem \ref{th2.b}-$(ii)$ we have,
\begin{equation}
\label{4.7}
\begin{aligned}
&\mu_{n,j}^{(i)}+\tilde{u}_{n,0}^{(i)}+2\log\frac{{\l_{n}^{(i)}} h(x_{n,j})}{8}+G_j^*(x_{n,j})\\
&=-\frac{2}{{\l_{n}^{(i)}} {h(x_{n,j}^{(i)})}}(\Delta \log h(x_{n,j}^{(i)}))(\mu_{n,j}^{(i)})^2e^{-\mu_{n,j}^{(i)}}
+O(\mu_{n,j}^{(i)}e^{-\mu_{n,j}^{(i)}}),
\quad i=1,2,
\end{aligned}
\end{equation}
which implies that,
\begin{equation}
\label{4.8}
\begin{aligned}
&\hat\mu_{n,j}^{(i)}+\log\frac{{\l_{n}^{(i)}} }{\int_{\Omega}he^{u_n^{(i)}}}+2\log\frac{h(x_{n,j})}{8}+G_j^*(x_{n,j})\\
&=-\frac{2}{{\l_{n}^{(i)}} {h(x_{n,j}^{(i)})}}(\Delta \log h(x_{n,j}^{(i)}))(\mu_{n,j}^{(i)})^2e^{-\mu_{n,j}^{(i)}}
+O(\mu_{n,j}^{(i)}e^{-\mu_{n,j}^{(i)}}),
\quad i=1,2,
\end{aligned}
\end{equation}
where we used $\tilde{u}_{n,0}^{(i)}=-\log\int_{\Omega}he^{u_n^{(i)}},~i=1,2.$ By \eqref{2.6} and Theorem \ref{th2.b}-(iii), we have
\begin{equation}
\label{4.9}
|x_{n,j}^{(1)}-x_{n,j}^{(2)}|=O(\sum_{i=1}^2\mu_{n,j}^{(i)}e^{-\mu_{n,j}^{(i)}})\quad\mathrm{and}\quad
|x_{n,j,*}^{(1)}-x_{n,j,*}^{(2)}|=O(\sum_{i=1}^2\mu_{n,j}^{(i)}e^{-\mu_{n,j}^{(i)}}).
\end{equation}
By \eqref{4.8}-\eqref{4.9}, and the fact that $\v_n^2={\frac{{\l_{n}^{(i)}} }{\int_{\Omega}he^{u_n^{(i)}}}},~i=1,2,$ we get
\begin{equation}
\label{4.10}
\hat\mu_{n,j}^{(1)}-\hat\mu_{n,j}^{(2)}=O(\sum_{i=1}^2\mu_{n,j}^{(i)}e^{-\mu_{n,j}^{(i)}}),
\end{equation}
which implies the first claim. We remark that the assumptions about $l(\bf{q})$ and $D(\bf{q})$ are not needed here since
$\v_n^2={\frac{{\l_{n}^{(i)}} }{\int_{\Omega}he^{u_n^{(i)}}}},~i=1,2$.
\medskip

$(ii)$ {Let us fix a small constant $r_0>0$.} In view of \eqref{2.7} and Theorem \ref{th2.a}-(a), we see that for $x\in B_{r_0}(q_j)$, it holds,
\begin{equation}
\label{4.11}
\begin{aligned}
\hat u_n^{(1)}-\hat u_n^{(2)}
=~&U_{n,j}^{(1)}-U_{n,j}^{(2)}+G_j^*(x_{n,j}^{(2)})-G_j^*(x_{n,j}^{(1)})+\log\frac{{\l_{n}^{(1)}}}{{\l_{n}^{(2)}}}+\eta_{n,j}^{(1)}-\eta_{n,j}^{(2)}\\
=~&U_{n,j}^{(1)}-U_{n,j}^{(2)}+G_j^*(x_{n,j}^{(2)})-G_j^*(x_{n,j}^{(1)})+\log\frac{{\l_{n}^{(1)}}}{{\l_{n}^{(2)}}}
+O(\sum_{i=1}^2(\mu_{n,j}^{(i)})^2e^{-\mu_{n,j}^{(i)}}).
\end{aligned}
\end{equation}
By the definition of $U_{n,j}^{(i)}$, we find that,
\begin{equation}
\label{4.12}
U_{n,j}^{(1)}-U_{n,j}^{(2)}+\log\frac{{\l_{n}^{(1)}}}{{\l_{n}^{(2)}}}=
2\log\dfrac{(1+\frac{h(x_{n,j}^{(2)})}{8}e^{\hat\mu_{n,j}^{(2)}}|x-x_{n,j,*}^{(2)}|^2)}
{(1+\frac{h(x_{n,j}^{(1)})}{8}e^{\hat\mu_{n,j}^{(1)}}|x-x_{n,j,*}^{(1)}|^2)}
+\hat\mu_{n,j}^{(1)}-\hat\mu_{n,j}^{(2)}.
\end{equation}
In view of \eqref{4.9} and $(i)$ above, we also see that,
\begin{equation}
\label{4.13}
\begin{aligned}
&h(x_{n,j}^{(2)})e^{\hat\mu_{n,j}^{(2)}}|x-x_{n,j,*}^{(2)}|^2-h(x_{n,j}^{(1)})e^{\hat\mu_{n,j}^{(1)}}|x-x_{n,j,*}^{(1)}|^2\\
&=O(e^{\hat\mu_{n,j}^{(1)}})\left(|x-x_{n,j,*}^{(1)}||x_{n,j,*}^{(1)}-x_{n,j,*}^{(2)}|+|x_{n,j,*}^{(1)}-x_{n,j,*}^{(2)}|^2\right)\\
&\quad+O(e^{\hat\mu_{n,j}^{(1)}})\left(|x-x_{n,j,*}^{(1)}|^2(|\hat\mu_{n,j}^{(1)}-\hat\mu_{n,j}^{(2)}|
+|x_{n,j}^{(1)}-x_{n,j}^{(2)}|)\right),
\end{aligned}
\end{equation}
which together with \eqref{4.10}-\eqref{4.12}, implies that,
\begin{equation}
\label{4.14}
\|\hat u_n^{(1)}-\hat u_n^{(2)}\|_{L^{\infty}({B_{r_0}(q_j)})}=O(\sum_{i=1}^2\hat\mu_{n,1}^{(i)}e^{-\frac{\hat\mu_{n,1}^{(i)}}{2}}).
\end{equation}

Next, we estimate $\hat u_n^{(1)}-\hat u_n^{(2)}$ in $\Omega\setminus\bigcup_{j=1}^mB_{r_0}(q_j)$. By \eqref{2.8} and Theorem \ref{th2.a}-(a), we have,
\begin{equation}
\label{4.15}
\begin{aligned}
\hat u_n^{(1)}-\hat u_n^{(2)}=~&\sum_{j=1}^m\left(\rho_{n,j}^{(1)}G(x,x_{n,j}^{(1)})-{\rho_{n,j}^{(2)}G(x,x_{n,j}^{(2)})}\right)
+\phi_{n}^{(1)}-\phi_{n}^{(2)}=o(\sum_{i=1}^2e^{-\frac{\hat\mu_{n,j}^{(i)}}{2}})
\end{aligned}
\end{equation}
for $x\in\Omega\setminus\bigcup_{j=1}^mB_{r_0}(q_j)$. Clearly \eqref{4.14} and \eqref{4.15} prove $(ii)$, and so the proof of Lemma \ref{le4.1} is completed.
\end{proof}

\medskip

Let us define,
\begin{equation}
\label{4.16}
\xi_n(x)=\frac{\hat u_n^{(1)}(x)-\hat u_n^{(2)}(x)}{\|\hat u_n^{(1)}-\hat u_n^{(2)}\|_{L^{\infty}(\Omega)}},\quad x\in\Omega,
\end{equation}
\begin{equation}
\label{4.17}
g_n^*(x)=\frac{h(x)}{\|\hat u_n^{(1)}-\hat u_n^{(2)}\|_{L^{\infty}(\Omega)}}\left(e^{\hat u_n^{(1)}(x)}-e^{\hat u_n^{(2)}(x)}\right),\quad x\in\Omega,
\end{equation}
and
\begin{equation}
\label{4.18}
c_n(x)=\frac{e^{\hat u_n^{(1)}}-e^{\hat u_n^{(2)}}}{{\hat u_n^{(1)}-\hat u_n^{(2)}}}=e^{\hat{u}_n^{(1)}}
\left(1+O(\|\hat u_n^{(1)}-\hat u_n^{(2)}\|_{L^{\infty}(\Omega)})\right),\quad x\in \Omega.
\end{equation}
It is not difficult to see that
\begin{equation}
\label{4.19}
\begin{cases}
\Delta\xi_n+g_n^*=\Delta\xi_n+hc_n\xi_n=0\quad &\mathrm{in}\quad \Omega,\\
\\
\xi_n=0\quad &\mathrm{on}\quad \partial\Omega.
\end{cases}
\end{equation}
Once more, it is worth to point out that $\xi_n=0$ on $\partial \Omega$ since $\v_n^2={\frac{{\l_{n}^{(i)}} }{\int_{\Omega}he^{u_n^{(i)}}}},~i=1,2.$\\

To study the behavior of $\xi_n$ in {$B_{\delta}(x_{n,j}^{(1)})$}, we set
\begin{equation}
\label{4.20}
\xi_{n,j}(z)=\xi_n(e^{-\frac{\hat\mu_{n,j}^{(1)}-\log8}{2}}z+x_{n,j}^{(1)}),\quad
|z|\leq \delta e^{\frac{\hat\mu_{n,j}^{(1)}-\log 8}{2}}\quad
\mathrm{for}\quad j=1,\cdots,m.
\end{equation}
\medskip

Our first estimate is about the limit of $\xi_{n,j}$.

\begin{lemma}
\label{le4.2}
There exists $b_{j,0},b_{j,1}$ and $b_{j,2}$ such that
\begin{equation}
\label{4.21}
\xi_{n,j}(z)\to \sum_{i=0}^3b_{j,i}\psi_{j,0}(z)\quad\mbox{in}\quad \m^0(\B^2),
\end{equation}
where
\begin{equation}
\label{4.22}
\psi_{j,0}(z)=\frac{1-h(q_j)|z|^2}{1+h(q_j)|z|^2},\quad \psi_{j,1}(z)=\frac{\sqrt{h(q_j)}z_1}{1+h(q_j)|z|^2},\quad
\psi_{j,2}(z)=\frac{\sqrt{h(q_j)}z_2}{1+h(q_j)|z|^2}.
\end{equation}
\end{lemma}
\begin{proof}
  From Lemma \ref{le4.1}, \eqref{2.6} and (a) in Theorem \ref{th2.a}, we have,
  {$$\Delta \xi_{n,j}+ 8h(x_{n,j}^{(1)})e^{-\hat\mu_{n,j}^{(1)}}e^{\hat{u}_n^{(1)}
  (e^{-\frac{\hat\mu_{n,j}^{(1)}-\ln8}{2}}z+x_{n,j}^{(1)})}\xi_{n,j}(1+o(1))=0,$$}
  {and}
\begin{equation}
\label{3.6}
{8h(x_{n,j}^{(1)})e^{-\hat\mu_{n,j}^{(1)}}e^{\hat{u}_n^{(1)}(e^{-\frac{\hat\mu_{n,j}^{(1)}-\ln8}{2}}z+x_{n,j}^{(1)})}\to \frac{8h(q_j)}{(1+ h(q_j)|z|^2)^2}\quad \mathrm{in}\quad \mathrm{C}_{\mathrm{loc}}^0(\B^2).
}\end{equation}
Therefore, since $|\xi_{n,j}|\leq1$ and because of the equation \eqref{4.19}, then we conclude that $\xi_{n,j}\to\xi_j$ in $\mathrm{C}_{\mathrm{loc}}^0(\B^2)$, where $\xi_j$ satisfies,
\begin{equation}
\label{3.7}
\Delta\xi_j+\frac{{8h(q_j)}}{(1+{h(q_j)}|z|^2)^2}\xi_j=0\quad\mathrm{in}\quad \B^2.
\end{equation}
By \cite[Proposition 1]{bp}, {there are constants $b_{j,i}$ such that}
\begin{equation*}
\xi_j=b_{j,0}\psi_{j,0}+b_{j,1}\psi_{j,1}+b_{j,2}\psi_{j,2},
\end{equation*}
which proves the lemma.
\end{proof}

\medskip

As discussed in the introduction,
we shall see in the next lemma that all the Fourier coefficients $b_{j,0}$ take the same value and that $\xi_n(x)$ will converge to such constant
(denoted by $b_0$) far away from the blow up points. Furthermore we will prove that $b_0=b_{j,1}=b_{j,2}=0$ for $j=1,\cdots,m.$

 {For any $r>0$, we denote by} \[{\Lambda_{n,j,r}^-=re^{-(\hat\mu_{n,j}^{(1)}-\log8)/2}, \ \textrm{ and}\  \Lambda_{n,j,r}^+=re^{(\hat\mu_{n,j}^{(1)}-\log8)/2},~j=1,\cdots,m.}\]
\begin{lemma}
\label{le3.2}
Let $\xi_n$ be defined in \eqref{4.16}, then we have,
\begin{itemize}
  \item [$(i)$] $b_{j,0}=b_0=0$ for $j=1,\cdots,m$ and
  $$\xi_n(x)=o_R(1)+o_n(1)\quad\mbox{in}\quad \m ^0(\om\setminus\bigcup_{j=1}^m B_{\Lambda_{n,j,R}^-}({x_{n,j}^{(1)}})~\mbox{as}~n\to+\infty,$$
  where $o_R(1)\to0$ as $R\to+\infty$ and $n\to0$ as $n\to+\infty.$
	
	\medskip
	
  \item [$(ii)$] $b_{j,1}=b_{j,2}=0$ for $j=1,\cdots,m.$
\end{itemize}
\end{lemma}
\medskip

Taking Lemma \ref{le3.2} for granted, then we can conclude the proof of Theorem \ref{th1.2}.
\bigskip

\noindent {\em Proof of Theorem \ref{th1.2}.} Let $x_n^*$ be a maximum point of $\xi_n$, then we have
\begin{equation}
\label{3.8}
|\xi_n(x_n^*)|=1.
\end{equation}
In view {of Lemma \ref{le3.2}}, we find that $\lim\limits_{n\to+\infty}x_n^*=q_j$ for some $j$ and
\begin{equation}
\label{3.9}
{\lim\limits_{n\to+\infty}e^{\frac{\hat\mu_{n,j}^{(1)}-\ln8}{2}}s_n=+\infty,\quad\mathrm{where}~s_n=|x_n^*-x_{n,j}^{(1)}|.}
\end{equation}
On the other hand, let
$$\tilde\xi_n(x)=\xi_n(s_nx+{x_{n,j}^{(1)}}).$$
From Theorem \ref{th2.a} and equation \eqref{4.19}, we can see that $\tilde\xi_n$ satisfies,
\begin{equation*}{
0=\Delta\tilde\xi_n+ s_n^2hc_n\tilde\xi_n=\Delta\tilde\xi_n
+\dfrac{ s_n^2h(x_{n,j}^{(1)})e^{\hat\mu_{n,j}^{(1)}}\tilde\xi_n\left(1+o(1)+O(s_n|x|)\right)}
{(1+\frac{ h(x_{n,j}^{(1)})}{8}e^{\hat\mu_{n,j}^{(1)}}|s_nx+x_{n,j}^{(1)}-x_{n,j,*}^{(1)}|^2)^2}.}
\end{equation*}
By using \eqref{3.8} we find that,
\begin{equation}
\label{3.10}
\left|\tilde\xi_n((x_n^*-{x_{n,j}^{(1)})}/s_n)\right|=|\xi_n(x_n^*)|=1.
\end{equation}
Next, in view of \eqref{3.9} and the fact that $|\tilde\xi_n|\leq1$, we see that $\tilde\xi_n\to\tilde\xi_0$ on any compact subset of $\B^2\setminus\{0\}$, where $\tilde\xi_n$ satisfies
\begin{equation}
\label{3.11}
\Delta\tilde\xi_n=0\quad\mathrm{in}\quad\B^2\setminus\{0\}.
\end{equation}
Since $|\tilde\xi_n|\leq 1$, then $x=0$ is a removable singularity. As a consequence,
$\tilde\xi_0$ is a bounded harmonic function in $\B^2$ and then from \eqref{3.10} we conclude that either
$\tilde\xi_0=1$ or $\tilde\xi_0=-1$. By using $e^{-{\frac{\hat\mu_{n,j}^{(1)}}{2}}}\ll s_n$ and $\lim\limits_{n\to\infty}s_n=0$,
we also find that,
\begin{equation}
\label{3.12}
|\xi_n(x)|\geq\frac12\quad\mathrm{for}\quad s_n\leq|x-{x_{n,j}^{(1)}}|\leq 2s_n,
\end{equation}
which is the desired contradiction to Lemma \ref{le3.2}. This fact concludes the proof of Theorem \ref{th1.2}.\hfill $\square$

\medskip

Next we prove Lemma \ref{le3.2} and divide the discussion into two parts:
\begin{enumerate}
  \item We first prove that $b_{j,0}=b_0=0$ for $j=1,\cdots,m$ by Green's identity and O.D.E. type arguments.
  \item Then we use {a suitably defined Pohozaev identity} to show that $b_{j,1}=b_{j,2}=0$ for $j=1,\cdots,m$.
\end{enumerate}
\medskip

\noindent {\em Proof of Lemma \ref{le3.2}-(i).} We recall that $\xi_n$ satisfies,
\begin{equation}
\label{3.13}
\Delta\xi_n+{he^{\hat u_n^{(1)}}\xi_n(1+o(1))}=0\quad\mathrm{in}\quad\Omega.
\end{equation}
It is easy to see that ${e^{\hat u_n^{(1)}}\to0}$ in $\m^0(\om\setminus\{q_1,\cdots,q_m\})$ {by \eqref{2.8} and Theorem 2B-$(ii)$.}
Since $\|\xi_n\|_{L^{\infty}(\Omega)}\leq1$, we see that $\xi_n\to\xi_0$ in $\m^0(\om\setminus\{q_1,\cdots,q_m\})$, where the points $\{q_1,\cdots,q_m\}$ are removable singular points of $\xi_0$ and
\begin{equation}
\label{3.14}
\Delta\xi_0=0\quad\mathrm{in}\quad\Omega.
\end{equation}
Therefore, in view of the Dirichlet boundary conditions, we find,
\begin{equation}
\label{3.15}
\xi_0=0\quad\mathrm{in}\quad \Omega,
\end{equation}
which implies that,
\begin{equation}
\label{3.16}
\xi_n\to {b_0=} 0\quad\mathrm{in}\quad \m^0(\om\setminus\{q_1,\cdots,q_m\}).
\end{equation}
Let us fix a small constant $\delta>0$ and set,
$\psi_{n,j}({x})=\frac{1-\frac{h({x}_{n,j}^{(1)})}{8}|{x}-{x}_{n,j}^{(1)}|^2e^{\hat\mu_{n,j}^{(1)}}}
{1+\frac{h({x}_{n,j}^{(1)})}{8}|{x}-{x}_{n,j}^{(1)}|^2e^{\hat\mu_{n,j}^{(1)}}}$.  For $d\in (0,\delta)$ and
in view of {\eqref{2.6}}, we find that,
\begin{equation*}\label{withpsi}\begin{aligned}
&\int_{\partial B_d({x}_{n,j}^{(1)})}\left(\psi_{n,j}\frac{\partial\xi_n}{\partial\nu}-\xi_n\frac{\partial \psi_{n,j}}{\partial \nu}\right)\mathrm{d}\sigma=\int_{B_d({x}_{n,j}^{(1)})}\left(\psi_{n,j}\Delta\xi_n-\xi_n\Delta \psi_{n,j}\right)\mathrm{d} {x}
\\&=\int_{B_d({x}_{n,j}^{(1)})}
\Bigg\{- \xi_n\psi_{n,j}{h}\left(\frac{e^{\hat{u}_n^{(1)}}-e^{\hat{u}_n^{(2)}}}{\hat{u}_n^{(1)}-\hat{u}_n^{(2)}}
\right)+\xi_n\psi_{n,j}{h}({x}_{n,j}^{(1)})e^{U_{n,j}^{(1)}}\Big(\frac{1+\frac{h({x}_{n,j}^{(1)})}{8}|{x}-{x}_{n,j,*}^{(1)}|^2
e^{\hat\mu_{n,j}^{(1)}}}{1+\frac{h({x}_{n,j}^{(1)})}{8}|{x}-{x}_{n,j}^{(1)}|^2e^{\hat\mu_{n,j}^{(1)}}}\Big)^{2}\Bigg\}
\mathrm{d} {x}\\&
=\int_{B_d({x}_{n,j}^{(1)})}\rho_n\xi_n\psi_{n,j}\Bigg\{-{h}e^{\tilde{u}_n^{(1)}}(1+O(|\hat{u}_n^{(1)}-\hat{u}_n^{(2)}|))
+h({x}_{n,j}^{(1)})e^{U_{n,j}^{(1)}}(1+O({ e^{-\frac{\hat\mu_{n,j}^{(1)}}{2}}}))\Bigg\}\mathrm{d} {x}\\
&=\int_{B_d({x}_{n,j}^{(1)})}\rho_n\xi_n\psi_{n,j}\Bigg\{-{h}e^{U_{n,j}^{(1)}+\eta_{n,j}^{(1)}+G_j^*({x})-G_j^*({x}_{n,j}^{(1)})}
(1+O(|\hat{u}_n^{(1)}-\hat{u}_n^{(2)}|))
\\&\quad+h({x}_{n,j}^{(1)})e^{U_{n,j}^{(1)}}(1+O({e^{-\frac{\hat\mu_{n,j}^{(1)}}{2}}}))\Bigg\}\mathrm{d} {x}.
\end{aligned}
\end{equation*}
Let $\overline{f}(z)=f(e^{-\frac{\mu_{n,j}^{(1)}}{2}}z+{x}_{n,j}^{(1)})$.
By a suitable scaling and by using Theorem 2A, we see that,
\begin{equation*}\label{withpsi1}
\begin{aligned}
&\int_{\partial B_d({x}_{n,j}^{(1)})}\left(\psi_{n,j}\frac{\partial\xi_n}{\partial\nu}-\xi_n\frac{\partial \psi_{n,j}}{\partial \nu}\right)\mathrm{d}\sigma
\\& =\int_{B_{\Lambda_{n,j,d}^{+}}(0)}\rho_n\overline{\xi_n}(z)\,\overline{\psi_{n,j}}(z)
\frac{ O(1)(e^{-\frac{\hat\mu_{n,j}^{(1)}}{2}}|z|+|\overline{\hat{u}_n^{(1)}}-\overline{\hat{u}_n^{(2)}}|+
{e^{-\frac{\hat\mu_{n,j}^{(1)}}{2}}}) }{(1+\frac{\rho_nh({x}_{n,j}^{(1)})}{8} |z+
e^{\frac{\mu_{n,j}^{(1)}}{2}}({x}_{n,j}^{(1)}-{x}_{n,j,*}^{(1)})|^2)^2}\mathrm{d} z.
\end{aligned}
\end{equation*}
In view of Lemma \ref{le4.1}-$(ii)$, we obtain
\begin{equation}
\label{withpsi2}
\begin{aligned}
&\int_{\partial B_d({x}_{n,j}^{(1)})}\left(\psi_{n,j}\frac{\partial\xi_n}{\partial\nu}-\xi_n\frac{\partial \psi_{n,j}}{\partial \nu}\right)
\mathrm{d}\sigma=o\left(\frac{1}{\hat\mu_{n,j}^{(1)}}\right).
\end{aligned}
\end{equation}
Let $\xi_{n,j}^*(r)=\int_0^{2\pi}\xi_n(r,\theta)\mathrm{d}\theta$, where $r=|x-x_{n,j}^{(1)}|$. Then \eqref{withpsi2} yields,
\begin{equation*}
\label{radial_eq}
(\xi_{n,j}^*)'(r)\psi_{n,j}(r)-\xi_{n,j}^*(r)\psi_{n,j}'(r)=\frac{o\left(\frac{1}{\hat\mu_{n,j}^{(1)}}\right)}{r},\,\forall r\in (\Lambda_{n,j,R}^{-}, {\delta}].
\end{equation*}
For any $R>0$ large enough and for any $r\in (\Lambda_{n,j,R}^{-}, {\delta})$, we also obtain that,
\[\psi_{n,j}(r)=-1+O\left(\frac{e^{-\hat\mu_{n,j}^{(1)}}}{r^2}\right),\ \ \psi_{n,j}'(r)=
O\left(\frac{{e^{-\hat\mu_{n,j}^{(1)}}}}{r^3}\right).\]
and so we conclude that,
\begin{equation}\label{derivative_xi}
  (\xi_{n,j}^*)'(r)=\frac{o\Big(\frac{1}{ \hat\mu_{n,j}^{(1)}}\Big)}{r}+O\left(\frac{{e^{-\hat\mu_{n,j}^{(1)}}}}{r^3}\right)
  \ \ \textrm{for all } \ \ r\in (\Lambda_{n,j,R}^{-}, {\delta}).
\end{equation}
Integrating \eqref{derivative_xi}, we obtain,
\begin{equation}
\label{compare_xi}
\begin{aligned}
\xi_{n,j}^*(r) =\xi_{n,j}^*(\Lambda_{n,j,R}^{-})+o(1)+o\Big(\frac{1}{ \hat\mu_{n,j}^{(1)}}\Big)R+O({R^{-2}})  \ \ \textrm{for all } \ \ r\in (\Lambda_{n,j,R}^{-}, {\delta}).
\end{aligned}
\end{equation}
By using Lemma \ref{le4.2}, we find,
\[\xi_{n,j}^*(\Lambda_{n,j,R}^{-})=-2\pi b_{j,0}+o_R(1)+o_n(1),\]
where $\lim\limits_{R\to+\infty}o_R(1)=0$ and $\lim\limits_{n\to+\infty}o_n(1)=0$ and then \eqref{compare_xi} shows that,
\begin{equation}
\label{outside_xi}
\xi_{n,j}^*(r)=-2\pi b_{j,0}+o_R(1)+o_n(1)(1+O(R)),\ \ \textrm{for all } \ \ r\in (\Lambda_{n,j,R}^{-},{\delta}).
\end{equation}
In view of \eqref{3.16}, we see that,
\begin{equation*}\label{compare_b}\xi_{n,j}^*= o_n(1) \ \ \textrm{in}\ \ C_{\textrm{loc}}(M\setminus\{q_1\cdots q_m\}),
\end{equation*} which implies that $b_{j,0}=0$ for $j=1,\cdots, m$. This fact, together with \eqref{3.16},
completes the proof of Lemma \ref{le3.2}-$(i)$.\hfill $\square$
\bigskip

For $j=1,\cdots,m$, let
\begin{equation}
\label{phij}
\phi_{n,j}(y)={\frac{\lambda_n^{(1)}}{m}}\left[(R(y,x_{n,j}^{(1)})-R(x_{n,j}^{(1)},x_{n,j}^{(1)}))
+\sum_{l\neq j}(G(y,x_{n,l}^{(1)})-G(x_{n,j}^{(1)},x_{n,l}^{(1)}))\right],
\end{equation}
\begin{equation}
\label{vnji}
v_{n,j}^{(i)}(y)={\hat{u}_n^{(i)}(y)}-\phi_{n,j}(y),\ \ \ i=1,2.
\end{equation}
Recall the definition of $\xi_n$ given before \eqref{4.17}. Our aim is to show that all $b_{j,i}=0$, see Lemma~\ref{le4.2}.
This is done by exploiting the following Pohozaev identity to derive a subtle estimate for $\xi_n$.
\begin{lemma}[\cite{ly2,BJLY}]
\label{le3.3} We have for $i = 1, 2$ and fixed small  {$r\in(0,\delta)$}, it holds,
\begin{equation}\label{3.27}\begin{aligned}
&\int_{\partial B_r(x_{n,j}^{(1)})}\langle\nu, D\xi_n\rangle D_iv_{n,j}^{(1)} +
\langle \nu, Dv_{n,j}^{(2)}\rangle D_i\xi_n\mathrm{d}\sigma
\\&-\frac{1}{2}\int_{\partial B_r(x_{n,j}^{(1)})}\langle D(v_{n,j}^{(1)}+v_{n,j}^{(2)}),D\xi_n\rangle
{\frac{(x-x_{n,j}^{(1)})_i}{|x-x_{n,j}^{(1)}|}}\mathrm{d}\sigma
\\=&-\int_{\partial B_r(x_{n,j}^{(1)})}h(x)\frac{e^{\hat{u}_n^{(1)}}-e^{\hat{u}_n^{(2)}}}
{\|\hat{u}_n^{(1)}-\hat{u}_n^{(2)}\|_{L^\infty(M)}}\frac{(x-x_{n,j}^{(1)})_i}{|x-x_{n,j}^{(1)}|}\mathrm{d}\sigma
\\&+\int_{B_r(x_{n,j}^{(1)})} h(x)\frac{e^{\hat{u}_n^{(1)}}-e^{\hat{u}_n^{(2)}}}
{\|\hat{u}_n^{(1)}-\hat{u}_n^{(2)}\|_{L^\infty(M)}}D_i(\phi_{n,j}+\log h)\mathrm{d} x.
\end{aligned}
\end{equation}
\end{lemma}
We will need an estimate about both sides of \eqref{3.27}. This is the outcome of rather delicate computations whose proof is given in the appendix.
\begin{lemma}
\label{le3.4}
For $j=1,\cdots,m$, we define
\begin{equation}
\label{3.30}
\quad R_{n,j}^*(x){=}\sum_{h=1}^2\partial_{ y_h}R(y,x)|_{{y=x_{n,j}^{(1)}}}b_{j,h}\tilde B_j,\quad G_{n,k}^*(x)=\sum_{h=1}^2\partial_{y_h}G(y,x)|_{y={x_{n,k}^{(1)}}}b_{k,h}\tilde B_k,
\end{equation}
where
\begin{equation*}
\tilde B_j=8\sqrt{\frac{2}{ h(x^{(1)}_{n,j})}}\int_{\B^2}\frac{|z|^2}{(1+|z|^2)^3}\mathrm{d}z.
\end{equation*}
Then
\begin{equation}
\label{3.31}
\mbox{R.H.S of}~\eqref{3.27}=\tilde B_j\left(e^{-\frac{\hat\mu^{(1)}_{n,j}}{2}}\sum_{h=1}^2D_{hi}^2(\phi_{n,j}+\log h)(x^{(1)}_{n,j})b_{j,h}\right)
+o(e^{-\frac{\hat\mu^{(1)}_{n,j}}{2}}),
\end{equation}
and
\begin{equation}
\label{3.32}
\mbox{L.H.S of}~\eqref{3.27}=-8\pi\left[e^{-\frac{\hat\mu^{(1)}_{n,j}}{2}}D_iR_{n,j}^*(x^{(1)}_{n,j})
%\sum_{h=1}^2\partial y_hR(y,x)|_{x=y=x^{(1)}_{n,j}}b_{j,h}\tilde B_j
+\sum_{k\neq j}e^{-\frac{\hat\mu^{(1)}_{n,k}}{2}}D_iG_{n,k}^*(x^{(1)}_{n,j})\right]+o(e^{-\frac{\hat\mu^{(1)}_{n,j}}{2}}).
\end{equation}
\end{lemma}

In view of these estimates we are now able to prove Lemma \ref{le3.2}-$(ii)$.
\medskip

\noindent {\em Proof of Lemma \ref{le3.2}-$(ii)$.} From Lemma \ref{le3.3}-\ref{le3.4}, we have for $i=1,2$
\begin{equation}
\label{3.33}
\begin{aligned}
&\tilde B_j\sum_{h=1}^2\left(e^{-\frac{\hat\mu^{(1)}_{n,j}}{2}}D_{hi}^2(\phi_{n,j}+\log h)(x^{(1)}_{n,j})b_{j,h}\right)\\
&=-8\pi\sum_{k\neq j}e^{-\frac{\hat\mu^{(1)}_{n,k}}{2}}\sum_{h=1}^2D_{x_i}\partial_{y_h}G(y,x)|_{(y,x)=(x^{(1)}_{n,j},x^{(1)}_{n,k})}b_{kh}\tilde B_k\\
&\quad-8\pi e^{-\frac{\hat\mu^{(1)}_{n,j}}{2}}\sum_{h=1}^2D_{x_i}\partial_{y_h}R(y,x)|_{(y,x)=(x^{(1)}_{n,j},x^{(1)}_{n,j})}b_{jh}\tilde B_j+o(e^{-\frac{\hat\mu^{(1)}_{n,j}}{2}}).
\end{aligned}
\end{equation}
Set
\begin{equation*}
{\bf b}=(\hat b_{1,1}\tilde B_1,\hat b_{1,2}\tilde B_1,\cdots,\hat b_{m,1}\tilde B_m,\hat b_{m,2}\tilde B_m),
\end{equation*}
where
$$\hat b_{kh}=\lim\limits_{n\to+\infty}(e^{\frac{\hat\mu^{(1)}_{n,j}-\hat\mu^{(1)}_{n,k}}{2}})b_{kh}.$$
Then by Theorem \ref{th2.b}-(iii) and sending $n$ to $+\infty$, we obtain that
\begin{equation}
\label{3.34}
D^2f_m(q_1,\cdots,q_m)\cdot{\bf b}=0,
\end{equation}
where $f_m$ is defined in \eqref{1.7}. By the non-degeneracy assumption $\mathrm{det}(D_{\Omega}^2f_m({\bf q}))\neq 0$, then we can immediately conclude that ${\bf b}=0$, i.e.,
$$b_{j,1}=b_{j,2}=0,\quad \forall j=1,\cdots,m.$$
This fact concludes the proof of Lemma \ref{le3.2}. \hfill $\square$

\bigskip
\section{Appendix}
This section is devoted to the proof of Lemma \ref{le3.4}. First of all, we prove an estimate which will be used later on.
\medskip
\begin{lemma}
\label{lea.1}
\begin{equation}
\label{a0.1}
\xi_n(x)=\sum_{j=1}^mA^{(1)}_{n,j}G(x^{(1)}_{n,j},x)+\sum_{j=1}^m\sum_{h=1}^2b_{n,j,h}\partial_{y_h}G(y,x)|_{y=x^{(1)}_{n,j}}+o(e^{-\frac{\hat\mu^{(1)}_{n,1}}{2}})
\end{equation}
holds in $C^1(\om\setminus\bigcup_{j=1}^mB_{\theta}(x^{(1)}_{n,j}))$ with suitable small constant $\theta$, where $\partial_{y_h}G(y,x)=\frac{\partial G(y,x)}{\partial y_h},~y=(y_1,y_2),$
\begin{equation*}
A^{(1)}_{n,j}=\int_{\Omega_j}g_n^*(y)\mathrm{d}y,\quad \mathrm{and}\quad b_{n,j,h}=e^{-\frac12\hat\mu^{(1)}_{n,j}}\frac{b_{j,h}8\sqrt{2}}{\sqrt{h(q_j)}}\int_{\mathbb{R}^2}\frac{|z|^2}{(1+|z|^2)^3}\mathrm{d}z.
\end{equation*}
\end{lemma}

\noindent {\bf Remark 4.1.} In fact, we can show that $A^{(1)}_{n,j}=o(e^{-\frac{\hat\mu^{(1)}_{n,j}}{2}}),j=1,\cdots,m$ by the same argument in {\cite[Lemma 4.4]{BJLY}, \cite[Lemma 5.4]{BJLY2}, \cite[(4.5)]{GOS}}. As a consequence, \eqref{a0.1} can be written as follows:
\begin{equation}
\label{a0.2}
\xi_n(x)=\sum_{j=1}^m\sum_{h=1}^2b_{n,j,h}\partial_{y_h}G(y,x)|_{y=x^{(1)}_{n,j}}+o(e^{-\frac{\hat\mu^{(1)}_{n,1}}{2}})
\quad\mathrm{in}\quad C^1(\om\setminus\bigcup_{j=1}^mB_{\theta}(x^{(1)}_{n,j})).
\end{equation}
\medskip

\begin{proof}
By the Green representation formula,
\begin{equation}
\label{a0.3}
\begin{aligned}
\xi_n(x)=&\int_{\Omega}G(y,x)g_n^*(y)\mathrm{d}y\\
=&\sum_{j=1}^mA^{(1)}_{n,j}G(x^{(1)}_{n,j},x)+\sum_{j=1}^m\int_{\Omega_j}(G(y,x)-G(x^{(1)}_{n,j},x))g_n^*(y)\mathrm{d}y.
\end{aligned}
\end{equation}
For $x\in\om\setminus\bigcup_{j=1}^mB_{\theta}(x^{(1)}_{n,j})$, we see from Theorem \ref{th2.a}-(a) and Theorem \ref{th2.b}-$(iv)$ that,
\begin{equation}
\label{a0.4}
\begin{aligned}
&\int_{\Omega_j}(G(y,x)-G(x^{(1)}_{n,j},x))g_n^*(y)\mathrm{d}y\\
&=\int_{B_r(x^{(1)}_{n,j})}\langle \partial_yG(y,x)|_{y=x^{(1)}_{n,j}},y-x^{(1)}_{n,j}\rangle g_n^*(y)\mathrm{d}y\\
&\quad +O(1)\left(\int_{B_r(x^{(1)}_{n,j})}\frac{|y-x^{(1)}_{n,j}|^2e^{\hat\mu^{(1)}_{n,j}}}{(1+e^{\hat\mu^{(1)}_{n,j}}|y-x^{(1)}_{n,j,*}|^2)^2}\mathrm{d}y\right)\
+O(e^{-\hat\mu^{(1)}_{n,j}})\\
&=\int_{B_r(x^{(1)}_{n,j})}\langle \partial_yG(y,x)|_{y=x^{(1)}_{n,j}},y-x^{(1)}_{n,j}\rangle g_n^*(y)\mathrm{d}y+O(e^{-\hat\mu^{(1)}_{n,j}})
\end{aligned}
\end{equation}
for suitable $r>0.$ In addition, we have
\begin{equation}
\label{a0.5}
\begin{aligned}
&\int_{B_r(x^{(1)}_{n,j})}\langle \partial_yG(y,x)|_{y=x^{(1)}_{n,j}},y-x^{(1)}_{n,j}\rangle g_n^*(y)\mathrm{d}y\\
&=16\sqrt{2}e^{-\frac{\hat\mu^{(1)}_{n,j}}{2}}\int_{B_{\Lambda_{n,j,r}}^+(0)}
\dfrac{\langle \partial_yG(y,x)|_{y=x^{(1)}_{n,j}},z\rangle  h(x^{(1)}_{n,j})\xi_{n,j}(z)}
{\big(1+h(x^{(1)}_{n,j}) |z+O(e^{-\frac{\hat\mu^{(1)}_{n,j}}{2}})|^2\big)^2}\mathrm{d}z+o(e^{-\frac{\hat\mu^{(1)}_{n,j}}{2}})
\end{aligned}
\end{equation}
By  Lemma \ref{le4.1}, we have for any $x\in\om\setminus\bigcup_{j=1}^mB_{\theta}(x^{(1)}_{n,j})$ we get
\begin{equation}
\label{a0.6}
\int_{B_r(x^{(1)}_{n,j})}\langle \partial_yG(y,x)|_{y=x^{(1)}_{n,j}},y-x^{(1)}_{n,j}\rangle g_n^*(y)\mathrm{d}y=\sum_{h=1}^2b_{n,j,h}\partial_{y_h}G(y,x)|_{y=x^{(1)}_{n,j}}+o(e^{-\frac{\hat\mu^{(1)}_{n,j}}{2}}).
\end{equation}
From \eqref{a0.3} and \eqref{a0.6}, we conclude that \eqref{a0.1} holds in $C^0(\om\setminus\bigcup_{j=1}^mB_{\theta}(x^{(1)}_{n,j}))$. The proof of the fact that \eqref{a0.1} holds in $C^1(\om\setminus\bigcup_{j=1}^mB_{\theta}(x^{(1)}_{n,j}))$ is similar and we omit the proof.
\end{proof}

Now we will provide the proof of Lemma \ref{le3.4}. We divide it into two parts.

\medskip
\noindent {\em Proof of Lemma \ref{le3.4}-\eqref{3.31}.} By Theorem \ref{th2.b}-$(iv)$, we have
\begin{equation}
\label{a1.1}
\int_{\partial B_r(x^{(1)}_{n,j})} he^{\hat u_n}\xi_n(1+o(1))\frac{(x-x^{(1)}_{n,j})_i}{|x-x^{(1)}_{n,j}|}\mathrm{d}\sigma=O(e^{-\hat\mu^{(1)}_{n,j}}).
\end{equation}
For the second term on the right hand side of \eqref{3.27}, it is not difficult to see that,
\begin{equation}
\label{a1.2}
\begin{aligned}
D_i(\phi_{n,j}+\log h)=~&D_i(\phi_{n,j}+\log h)(x^{(1)}_{n,j})+\sum_{h=1}^2D_{ih}^2(\phi_{n,j}+\log h)(x^{(1)}_{n,j})(x-x^{(1)}_{n,j})_h\\
&+O(|x-x^{(1)}_{n,j}|^2).
\end{aligned}
\end{equation}
Since ${\bf q}$ is a critical point of $f_m$, by \eqref{2.6} and Theorem \ref{th2.b}-$(iii)$-$(vi)$, we find that,
\begin{equation}
\label{a1.3}
D_i(\phi_{n,j}+\log h)(x^{(1)}_{n,j})=D_i(G_{j}^*+\log h)(x^{(1)}_{n,j})+O(\hat\mu^{(1)}_{n,j}e^{-\hat\mu^{(1)}_{n,j}})=O(\hat\mu^{(1)}_{n,j}e^{-\hat\mu^{(1)}_{n,j}}).
\end{equation}
From   \eqref{a1.2} and \eqref{a1.3} we have,
\begin{equation}
\label{a1.4}
\begin{aligned}
&\int_{B_r(x^{(1)}_{n,j})} he^{\hat u_n}\xi_n(1+o(1))D_i(\phi_{n,j}+\log h)\mathrm{d}x\\
&=\int_{B_{\Lambda_{n,j,r}^+}(0)}\frac{16\sqrt{2}h(x^{(1)}_{n,j})\xi_n(x^{(1)}_{n,j}+e^{-\frac{(\hat\mu^{(1)}_{n,j}-\log8)}{2}}z)}{(1+  h(x^{(1)}_{n,j}) |z|^2)^2}
\left(\sum_{h=1}^2D_{hi}^2(\phi_{n,j}+\log h)(x^{(1)}_{n,j})e^{-\frac{\hat\mu^{(1)}_{n,j}}{2}}z_h\right)dz\\
&\quad+o(e^{-\frac{\hat\mu^{(1)}_{n,j}}{2}}),
\end{aligned}
\end{equation}
and then combined with Lemma \ref{le4.1}, we conclude that
\begin{equation}
\label{a1.5}
\begin{aligned}
&\int_{B_r(x^{(1)}_{n,j})} he^{\hat u_n}\xi_n(1+o(1))D_i(\phi_{n,j}+\log h)\mathrm{d}x\\
&=\int_{B_{\Lambda_{n,j,r}^+}(0)}\frac{8\sqrt{2}h(x^{(1)}_{n,j})\sqrt{ h(x^{(1)}_{n,j}) }|z|^2}{(1+  h(x^{(1)}_{n,j}) |z|^2)^3}
\mathrm{d}z
\left(\sum_{h=1}^2b_{j,h}D_{hi}^2(\phi_{n,j}+\log h)(x^{(1)}_{n,j})e^{-\frac{\hat\mu^{(1)}_{n,j}}{2}}\right)\\
&\quad+o(e^{-\frac{\hat\mu^{(1)}_{n,j}}{2}})\\
&=\tilde B_j\left(\sum_{h=1}^2b_{j,h}D_{hi}^2(\phi_{n,j}+\log h)(x^{(1)}_{n,j})e^{-\frac{\hat\mu^{(1)}_{n,j}}{2}}\right)+o(e^{-\frac{\hat\mu^{(1)}_{n,j}}{2}}).
\end{aligned}
\end{equation}
Clearly \eqref{3.31} follows from \eqref{a1.1} and \eqref{a1.5}. \hfill $\square$

\medskip
\noindent {\em Proof of Lemma \ref{le3.4}-\eqref{3.32}.} By the definition of $G_{n,k}^*(x)$, we have
\begin{equation*}
\Delta G_{n,k}^*(x)=0\quad \mathrm{in}\quad B_r(x^{(1)}_{n,j})\setminus B_{\theta}(x^{(1)}_{n,j}),\quad \forall\theta\in(0,r).
\end{equation*}
Then for any $x\in B_r(x^{(1)}_{n,j})\setminus B_{\theta}(x^{(1)}_{n,j})$, we have
\begin{equation}
\label{a2.1}
\begin{aligned}
0=~&\Delta G_{n,k}^*D_i\log\frac{1}{|x-x^{(1)}_{n,j}|}+\Delta\log\frac{1}{|x-x^{(1)}_{n,j}|}D_iG_{n,k}^*\\
=~&\mathrm{div}\Bigg(\nabla G_{n,k}^*D_i\log\frac{1}{|x-x^{(1)}_{n,j}|}+\nabla\log\frac{1}{|x-x^{(1)}_{n,j}|}D_iG_{n,k}^*
\\~&-\nabla G_{n,k}^*\cdot\nabla\log\frac{1}{|x-x^{(1)}_{n,j}|}e_i\Bigg),
\end{aligned}
\end{equation}
where $e_i=\frac{x_i}{|x|},~i=1,2$. It implies that
\begin{equation}
\label{a2.2}
\int_{\partial B_r(x^{(1)}_{n,j})}\frac{\nabla_iG_{n,k}^*}{|x-x^{(1)}_{n,j}|}\mathrm{d}\sigma
=\int_{\partial B_{\theta}(x^{(1)}_{n,j})}\frac{\nabla_iG_{n,k}^*}{|x-x^{(1)}_{n,j}|}\mathrm{d}\sigma.
\end{equation}
In view of \eqref{a0.2}, we have
\begin{equation}
\label{a2.3}
\xi_n(x)=\sum_{k=1}^me^{-\frac{\hat\mu^{(1)}_{n,k}}{2}}G_{n,k}^*(x)+o(e^{-\frac{\hat\mu^{(1)}_{n,j}}{2}})\quad \mathrm{in}\quad \mathrm{C}^1(B_r(x^{(1)}_{n,j})\setminus B_{\theta}(x^{(1)}_{n,j})).
\end{equation}
By using Theorem \ref{th2.b}, we find that,
\begin{equation*}
\frac{\l_n^{(1)}}{m}=\rho^{(1)}_{n,j}+O(\hat\mu^{(1)}_{n,j}e^{-\hat\mu^{(1)}_{n,j}})=8\pi+O(\hat\mu^{(1)}_{n,j}e^{-\hat\mu^{(1)}_{n,j}}),
\end{equation*}
which together with Theorem \ref{th2.a}-(b), implies that for $i=1,2$,
\begin{equation}
\label{a2.4}
\begin{aligned}
\nabla v^{(i)}_{n,j}(x)=~&\nabla (\tilde u_n^{(i)}-\phi_{n,j})=\nabla\left(\tilde u_n^{(i)}-\frac{\l_n^{(1)}}{m}R(x,x^{(1)}_{n,j})-\frac{\l_n^{(1)}}{m}\sum_{k\neq j}G(x,x^{(1)}_{n,j})\right)\\
=~&\nabla\left(\tilde u_n^{(i)}-\frac{\l_n^{(1)}}{m}\sum_{k=1}^m G(x,x^{(1)}_{n,k})-\frac{\l_n^{(1)}}{2\pi m}\log|x-x^{(1)}_{n,j}|\right)\\
=~&\nabla\phi_n^{(i)}-4\frac{(x-x^{(1)}_{n,j})}{|x-x^{(1)}_{n,j}|^2}+o(e^{-\frac{\hat\mu^{(1)}_{n,j}}{2}})\\
=~&-4\frac{(x-x^{(1)}_{n,j})}{|x-x^{(1)}_{n,j}|^2}+o(e^{-\frac{\hat\mu^{(1)}_{n,j}}{2}})
\quad \mathrm{in}\quad  \mathrm{C}^1(B_r(x^{(1)}_{n,j})\setminus B_{\theta}(x^{(1)}_{n,j})).
\end{aligned}
\end{equation}
By using $D_iD_h(\log|z|)=\frac{\delta_{ih}|z|^2-2z_iz_h}{|z|^4}$, we see that,
\begin{equation}
\label{a2.5}
\int_{\partial B_{\theta}(x^{(1)}_{n,j})}\frac{\nabla_i G_{n,j}^*}{|x-x^{(1)}_{n,j}|}\mathrm{d}\sigma
=2\pi D_i\sum_{h=1}^2\partial_{y_h}R(y,x)|_{x=y=x^{(1)}_{n,j}}b_{j,h}\tilde B_j+o_{\theta}(1).
\end{equation}
From \eqref{a2.2}-\eqref{a2.5}, we get for any $\theta\in(0,r),$
\begin{equation}
\label{a2.6}
\begin{aligned}
&\mbox{L.H.S. of}~\eqref{3.27}\\&=-8\pi\left[e^{-\frac{\hat\mu^{(1)}_{n,j}}{2}}D_iR_{n,j}^*(x^{(1)}_{n,j})
%\sum_{h=1}^2\partial y_hR(y,x)|_{x=y=x^{(1)}_{n,j}}b_{j,h}\tilde B_j
+\sum_{k\neq j}e^{-\frac{\hat\mu^{(1)}_{n,k}}{2}}D_iG_{n,k}^*(x^{(1)}_{n,j})\right]+o_{\theta}(1)e^{-\frac{\hat\mu^{(1)}_{n,j}}{2}},
\end{aligned}
\end{equation}
where $o_{\theta}(1)\to0$ as $\theta\to0$. This fact concludes the proof of Lemma \ref{le3.4}. \hfill $\square$

\end{document}